\documentclass[classic,watermark,histinit,frontspecie,plain]{article} 

\usepackage{xcolor}

\usepackage[scientific-notation=true]{siunitx}
\usepackage[english]{babel}

\usepackage{graphicx}

\usepackage{color}
\graphicspath{{figures/}}

\usepackage{amsmath,amsthm,amssymb}
\usepackage{amscd}

\usepackage{titlesec}
\usepackage{setspace}
\usepackage{youngtab}
\usepackage{listings}
\usepackage{enumerate}
\usepackage{amssymb}
\usepackage{amsthm}
\usepackage{syntonly}
\usepackage{fancyhdr}
\usepackage{romannum}
\usepackage{cancel}
\usepackage{siunitx}
\usepackage{slashed}
\usepackage{bbm}
\usepackage{import}
\usepackage{relsize}
\usepackage{enumerate}
\usepackage{lettrine}
\usepackage[shortlabels]{enumitem}
\usepackage{subcaption}
\usepackage{multirow}
\usepackage{gensymb}
\usepackage{fancyhdr}
\usepackage{dsfont}
\usepackage{float}
\usepackage{fancyhdr}

\usepackage[english]{babel}

\usepackage{lipsum}

\usepackage{tabularx}

\makeatletter

\newcommand*\mybackmatter{%
	\startcontents
	\phantomsection
	\addcontentsline{toc}{part}{}%
	\@endpart}
\makeatother
\definecolor{tiffanyblue}{rgb}{0.3, 0.7, 0.9}
\definecolor{prussianblue}{rgb}{0.0, 0.19, 0.33}
\definecolor{rosewood}{rgb}{0.4, 0.0, 0.04}
\definecolor{tealblue}{rgb}{0.21, 0.46, 0.53}
\definecolor{frenchblue}{rgb}{0.0, 0.45, 0.73}
\definecolor{oceanboatblue}{rgb}{0.0, 0.47, 0.75}
\definecolor{coolblack}{rgb}{0.0, 0.18, 0.39}
\definecolor{RedViolet}{rgb}{0.78, 0.08, 0.52}
	\definecolor{richmaroon}{rgb}{0.69, 0.19, 0.38}
		\definecolor{rose}{rgb}{1.0, 0.0, 0.5}
			\definecolor{ruby}{rgb}{0.88, 0.07, 0.37}
			\definecolor{vividcerise}{rgb}{0.85, 0.11, 0.51}
\usepackage{minitoc}
\usepackage{etoc}
\usepackage{blindtext}

\usepackage{titlesec}
\titleclass{\part}{top}

\usepackage[utf8]{inputenc}
\usepackage{array}
\usepackage{lmodern}
\usepackage[T1]{fontenc}

\usepackage{shapepar}

\newcommand{\RomanNumeralCaps}[1]
{\MakeUppercase{\romannumeral #1}}
\newtheorem{conjecture}{Conjecture}
\newtheorem{definition}{Definition}
\newtheorem{problem}{Problem} 
\newtheorem{theorem}{Theorem}
\newtheorem{lemma}{Lemma}
\newtheorem{proposition}{Proposition}
\newtheorem{corollary}{Corollary}
\newtheorem{claim}{Claim}[theorem]
\newtheorem*{claim*}{Claim}

\begin{document}
	\title{Remarks on the subdivisions of bispindles and two-blocks cycles in highly chromatic digraphs}

\maketitle
\begin{center}\author{Darine AL MNINY \footnote[1]{Camille Jordan Institute, Claude Bernard University - Lyon 1, France. (mniny@math.univ-lyon1.fr)}$^{,}$ \footnote[2]{KALMA Laboratory, Department of Mathematics, Faculty of Sciences I, Lebanese University, Beirut - Lebanon. (darine.mniny@liu.edu.lb)},   Salman GHAZAL \footnote[3]{ Department of Mathematics, Faculty of Sciences I, Lebanese University, Beirut - Lebanon. (salman.ghazal@ul.edu.lb)}$^{,}$ \footnote[4]{ Department of Mathematics and Physics, School of Arts and Sciences, Beirut International University, Beirut - Lebanon. (salman.ghazal@liu.edu.lb)}} \end{center}

\begin{abstract}
	\noindent A  $(2+1)$-bispindle $B(k_1,k_2;k_3)$ is the union of two  $xy$-dipaths of respective lengths $k_1$ and $k_2$, and one $yx$-dipath of length $k_3$, all these dipaths being pairwise internally disjoint. Recently, Cohen et al. conjectured that, for every positive integers $k_1, k_2, k_3$, there is an integer $g(k_1, k_2, k_3)$ such that  every strongly connected digraph not containing  subdivisions of $B(k_1, k_2; k_3)$ has a chromatic number at most $g(k_1, k_2, k_3)$, and they proved it only for the case where $k_2=1$. For Hamiltonian digraphs, we prove Cohen et al.'s conjecture, namely $g(k_1, k_2, k_3)\leq 4k$, where $k=max\{k_1, k_2, k_3\}$. A   two-blocks cycle $C(k_1,k_2)$ is the union of two internally disjoint  $xy$-dipaths of length $k_1$ and $k_2$ respectively.  Addario et al. asked if the chromatic number of strong digraphs not containing subdivisions of a two-blocks cycle $C(k_1,k_2)$ can be bounded from above  by $O(k_1+k_2)$, which remains an open problem.  Assuming that $k=max\{k_1,k_2\}$, the best reached upper bound, found  by Kim et al., is $12k^2$. In this article, we conjecture that this bound can be slightly improved to $4k^2$ and we confirm our conjecture for some particular cases. Moreover, we provide a positive answer to Addario et al.'s question for the class of digraphs having a Hamiltonian directed path. 
	\end{abstract}
\section{Introduction}
\bigskip
Throughout this paper, all graphs  are considered to be \textit{simple}, that is, there are no loops and no multiple edges. By giving an orientation to each edge of a graph $G$, the obtained oriented graph is called a \textit{digraph}. Reciprocally, the graph obtained from a digraph $D$ by ignoring the directions of its arcs is called the \textit{underlying graph} of $D$, and denoted by $G(D)$.
The \textit{chromatic number} of a digraph $D$, denoted by $\chi(D)$, is the chromatic number of its underlying graph. Let $\cal{D}$ be a class of digraphs. The chromatic number of $\cal{D}$, denoted by $\chi(\cal D$$)$, is the smallest
integer $k$ such that $\chi(\cal D$$) \leq k$ for all $D \in \cal D$ or $+\infty$ if no such $k$ exists. By convention, if $\cal{D}$ $ = \emptyset$, 
then $\chi(\cal D$$)=0$. If $\chi(\cal D$$) \neq +\infty$, we say that $\cal{D}$ has a bounded chromatic number. \medbreak

\noindent  A \textit{directed path}, or simply a \textit{dipath}, is an oriented path where all the arcs are oriented in the same direction from the initial vertex towards the terminal vertex.  A classical result due to Gallai and Roy  \cite{Gallai, Roy} is the following: 
 \begin{theorem}\label{directedpath}(Gallai, 1968; Roy, 1967) Let   $k$ be a non-negative integer and $D$ be a digraph whose chromatic number is at least $k$. Then $D$ contains a directed path of order $k$.\end{theorem}
\noindent This raises the following question:
\begin{problem}\label{problem1} Which are the  digraph classes  $\cal{D}$ such that every digraph with  chromatic number at least $k$ contains an element of  $\cal{D}$ as a subdigraph?\end{problem}

 \noindent  Denoting by $Forb( \cal{D}$$)$ the class of digraphs that do not contain an element of a class of digraphs $\cal{D}$ as a subdigraph, the above question can be restated in terms of $Forb(\cal{D}$$)$ as follows: \textit{Which are the digraph classes $\cal{D}$ such that $\chi(Forb( \cal{D}$$))<+\infty$?} Due  to a famous theorem of Erd\"{o}s \cite{girth} which guarantees the existence  of graphs  with  arbitrarily  high  girth  and arbitrarily high  chromatic number,  if $H$ is a digraph containing an  oriented    cycle,  there  exist  digraphs  with  arbitrarily  high chromatic number with no subdigraph isomorphic to $H$. This means that the only possible candidates to  generalize  Theorem  \ref{directedpath}  are  the oriented  trees.  In this context,  Burr \cite{burr}  proved that the chromatic number of $Forb(T)$  for every oriented tree $T$ of order $k$  is at most $k^2-2k$ and he conjectured that this upper bound can be improved to $2k-3$: 
 
 \begin{conjecture}(Burr, 1980)\label{burrconj}
 	Let $k$ be a positive integer and let $T$ be an oriented tree of order $k$. Then the chromatic number of $Forb(T)$ is equal to $2k-3$. \end{conjecture}
 
 \noindent The best known upper bound, found by   Addario-Berry et al. \cite{tree}, is  $k^2/2 - k/2$.  However, for oriented paths with two blocks (\textit{blocks} are maximal directed subpaths), the best possible upper bound is known. Assuming that an oriented path $P$ has two blocks of lengths $k_1$ and $k_2$ respectively, we say that $P$ is a two-blocks path and we write $P=P(k_1,k_2)$.
\begin{theorem}(Addario-Berry et al. \cite{2path}, 2007)\label{path} Let $k_{1}$ and $k_{2}$ be positive integers such that $k_{1}+k_{2} \geqslant 3$. Then $Forb(P(k_{1},k_{2}))$ has a chromatic number equal to $k_{1}+k_{2}$, for every two-blocks path $P(k_{1},k_{2})$. \end{theorem}

\noindent Recall that a \textit{subdivision} of a digraph $H$ is a digraph $H'$  obtained from $H$ by replacing each arc $(x,y)$  by an $xy$-dipath of length at least $1$. A digraph  $D$ is said to be \textit{$H$-subdivision-free} if it contains no subdivisions of $H$ as a subdigraph. Inspired by the previous researches,  Cohen et al. \cite{cohen1}  asked about the existence of subdivisions of oriented cycles in highly chromatic digraphs.  In other words, denoting by $S$-$Forb(C)$ the class of digraphs that do not contain  subdivisions of a given oriented cycle $C$ as subdigraphs, Cohen et al. asked if  the chromatic number of $S$-$Forb(C)$ can be bounded.  In the same article, Cohen et al.  provide a negative answer to their question by  proving  a stronger theorem based on a construction  built by   Erd\"os and Lov\'{a}sz \cite{alm} that induces the existence of  hypergraphs with high girth and large chromatic number: \begin{theorem}(Cohen et al., 2018) For any positive integers $b,c$, there exists an acyclic digraph $D$ with $\chi(D) \geqslant c$ in which all oriented cycles have more than $b$ blocks. \end{theorem} 

\noindent On the other hand, restricting Cohen et al.'s question  on the class of  strongly connected  digraphs may lead to dramatically different results.  A digraph $D$ is said to be \textit{strongly connected}, or simply \textit{strong}, if for any two vertices $x$ and $y$ of $D$, there is a directed path from $x$ to $y$. A \textit{directed cycle}, or simply a \textit{circuit}, is an oriented cycle whose all arcs have the same orientation. An example is provided by a famous result of Bondy \cite{bondy} that dates back to 1976: \textit{Every strong digraph $D$ contains a directed cycle of length at least $\chi(D)$}. In other words, denoting by $\cal S$ the class of strong digraphs, the circuit $C^{+}_{k}$ of length $k$  satisfies $\chi(S$-$Forb(C^{+}_{k})\cap\cal S)=$ $k-1$. \medskip

\noindent Since any directed cycle of length at least $k$ can be seen as a subdivision of $C^{+}_{k}$, Cohen et al. \cite{cohen1} conjectured in 2018 that Bondy's theorem can be extended to all oriented cycles: 

\begin{conjecture}\label{conjcycle}(Cohen et al., 2018) For every oriented cycle $C$, there exists a constant $f(C)$ such that every strongly connected digraph with chromatic number at least $f(C)$ contains a subdivision of $C$. \end{conjecture}

\noindent For two positive integers $k_1$ and $k_2$, a \textit{cycle with two blocks} $C(k_1,k_2)$ is an oriented cycle which consists of  two internally disjoint directed paths of lengths $k_1$ and $k_2$ respectively. In their article,  Cohen et al. \cite{cohen1} proved Conjecture \ref{conjcycle}  for the case of  two-blocks cycles. More precisely, they showed that the chromatic number of strong digraphs with no subdivisions of a two-blocks cycle $C(k_{1},k_{2})$ is bounded from above by  $O((k_{1}+k_{2})^4)$:

\begin{theorem}	Let $k_{1}$ and $k_{2}$ be positive integers such that $k_{1} \geqslant k_{2} \geqslant 2$ and $k_{1} \geqslant 3$, and let $D$ be a digraph in $S$-$Forb(C(k_{1},k_{2})) \cap \cal S$. Then the chromatic number of $D$ is at most  $(k_{1}+k_{2}-2)(k_{1}+k_{2}-3)(2k_{2}+2)(k_{1}+k_{2}+1)$.\end{theorem} 

\noindent  More recently, this bound was  improved by  Kim et al. \cite{kim} as follows: 
\begin{theorem} (Kim et al., 2018) Let $k_{1}$ and $k_{2}$ be positive integers such that $k_{1} \geqslant k_{2} \geqslant 1$ and $k_{1} \geqslant 2$,  and let $D$ be a digraph in $S$-$Forb(C(k_{1},k_{2})) \cap \cal S$. Then the chromatic number of $D$ is at most $2 (2k_1-3)(k_1+2k_2-1)$. \end{theorem}
 
\noindent  As a key step, Kim et al. studied in the  same article the existence of two-blocks cycles in Hamiltonian digraphs (a digraph $D$ is said to be \textit{Hamiltonian} if it contains a Hamiltonian directed cycle, that is, a directed cycle passing through all the vertices of $D$),  and they were able to reach a linear upper bound for the chromatic number of such digraphs containing no subdivisions of a given  two-blocks cycle. Denoting by $\cal H$ the class of Hamiltonian digraphs, these authors proved precisely the following:

\begin{theorem}Let $k_1$ and $k_2$ be positive integers such that $k_1+k_2 \geq 3$, and let $D$ be a digraph in $S$-$Forb(C(k_{1},k_{2})) \cap \cal H$. Then  $\chi(D) \leq k_1+k_2$.\end{theorem}

\noindent  In \cite{2path}, Addario et al. asked if  the upper bound of strongly connected digraphs having no subdivisions of $C(k_1,k_2)$ can be improved to $O(k_1 + k_2)$, which  remains an  open problem. In this article, we conjecture that the chromatic number of  digraphs having a spanning out-tree (every strong digraph contains a spanning out-tree) without subdivisions of a two-blocks cycle $C(k_{1},k_{2})$ may be  improved to $4k^2$, where $k= max$ $ \{k_1,k_2\}$,  we prove our conjecture for some particular cases and we introduce an approach upon which  the interested reader may build to overcome the problem.    As a key step, we consider $C(k_1,k_2)$-subdivision-free digraphs having a Hamiltonian directed path, and we bound from above  the chromatic number of such digraphs by $3(k-1)$. This intermediate result will be used in the improvement mentioned before, replacing the complicated and technical proof in \cite{kim} based on Bondy's theorem \cite{bondy} which describes the strong digraphs structural properties, by an elementary  one based only on the simple notion of maximal out-tree. \medbreak

\noindent Since every tournament contains a Hamiltonian directed path (see \cite{camion}),  our work  on the chromatic number of $C(k_1,k_2)$-subdivision-free digraphs having a Hamiltonian directed path leads us to wonder about the chromatic number of tournaments having no subdivisions of a given oriented cycle formed of $t$ blocks with $t>2$ .   Our question is a weak version of Rosenfeld's conjecture which predicts that every tournament of order $ n \geq 3$ contains every oriented Hamiltonian cycle except possibly the directed one (see \cite{Rosenfled}). The best known result, due to Havet \cite{Havet}, asserts that Rosenfeld's conjecture is true for  every tournament of order $n \geq 68$.   In what follows, we denote by $C(k_{1}, k_{2}, ..., k_{t})$ the oriented cycle $C$ having $t$ blocks of consecutive lengths $k_{1}, k_{2}, ..., k_{t}$. In this case, we say that $C$ is a $t$-blocks cycle.  According to the definition of a block and a subdivision, note  that $t$   must be an even integer and that a subdivision of a $t$-blocks cycle is also a $t$-blocks cycle.  In this article,   we prove  that every tournament of order $m+\sum_{i=1}^{2m} k_i$ contains a subdivision $H$ of the oriented cycle  $C(k_1,k_2,...,k_{2m})$ such that the lengths of  at least $m$ blocks of $H$ are exactly the same as the lengths of the  corresponding ones of $C$, given that $k_i+k_{i+1} \geq 3$ for all $i \in \{1, 3, ..., 2m-1\}$ and $m \geq 2$. \medbreak 

\noindent A \textit{$p$-spindle} is the union of $p$ internally disjoint $xy$-dipaths for some vertices $x$ and $ y$. In this case, $x$ is the \textit{tail} of the spindle and $y$ is its \textit{head}.   A \textit{$(p+q)$-bispindle} is the internally disjoint union of a $p$-spindle with tail $x$ and head $y$ and a $q$-spindle with tail $y$ and head $x$.  In other words, it is the union of $p$ $xy$ -dipaths and $q$ $yx$-dipaths, all of these dipaths being pairwise internally disjoint. Since directed cycles and two-blocks cycles can be seen as $(1+1)$-bispindles and $2$-spindles respectively,  Cohen et al. \cite{cohen2} asked about the existence of spindles and bispindles in strong digraphs with large chromatic number. First, they pointed the existence of strong digraphs with large chromatic number that do not contain neither $3$-spindle nor $(2+2)$-bispindle.  Undoubtedly, this result guides them to focus  in their study on the existence of $(2+1)$-bispindles in strong digraphs. Denoting by $B(k_{1},k_{2};k_{3})$ the $(2+1)$-bispindle formed by the internally disjoint union of two $xy$-dipaths, one of length $k_{1}$ and the other of length  $k_{2}$, and one $yx$-dipath of length $k_{3}$, Cohen et al. \cite{cohen2} conjectured  the following:

\begin{conjecture}(Cohen et al., 2017)\label{bispindle} Let $D$ be a digraph in $S$-$Forb(B(k_{1},k_{2};k_{3})) \cap \cal S$. Then there exists a function $g:\mathbb{N}^{3} \longrightarrow \mathbb{N}$ such that $\chi(D) \leq g(k_{1},k_{2},k_{3})$.\end{conjecture}

\noindent In the same paper, Cohen et al. confirmed  their  conjecture for $B(k_1,1;k_3)$ and they attained a  better bound for the case where  $k_{1}$ is arbitrary and  $k_{2}=k_{3}=1$. The upper  bound that Cohen et al. provided for  the chromatic number of digraphs in $S$-$Forb(B(k_{1},1;k_{3})) \cap \cal S$ is huge and certainly not the best possible.  In this article, we contribute to Conjecture \ref{bispindle} by  affirming it  for  the class of Hamiltonian digraphs.       
\section{Preliminaries and definitions} 
A graph $G$ is said to be \textit{$d$-degenerate}, if any  subgraph of $G$ contains a vertex having at most $d$ neighbors. Using an inductive argument, we may easily remark that any $d$-degenerate graph is $(d+1)$-colorable.\medbreak

\noindent Given two graphs $G_{1}$ and $G_{2}$, $G_{1} \cup G_{2}$ is defined to be the graph whose vertex-set is $V(G_{1}) \cup V(G_{2})$ and whose edge-set is $E(G_{1}) \cup E(G_{2})$. The next lemma will be useful for the coming proofs:
\begin{lemma}\label{far} 
	$\chi(G_{1} \cup G_{2}) \leqslant \chi(G_{1}) \times \chi(G_{2})$ for any two graphs  $G_{1}$ and $G_{2}$.
\end{lemma}
\begin{proof}
	For $i \in \{1,2\}$, let $\phi_{i}: V(G_{i}) \longrightarrow \{1, 2, ..., \chi(G_{i})\}$ be a proper $\chi(G_{i})$-coloring of $G_{i}$. 
	Define $\psi$, the coloring of  $V(G_{1} \cup G_{2})$, as follows:
	\begin{align*}
	\psi(x) = \left\{ \begin{array}{cc} 
	(\phi_{1}(x), 1) & \hspace{5mm}  x \in V(G_{1}) \setminus V(G_{2});  \\
	(\phi_{1}(x), \phi_{2}(x)) & \hspace{5mm} x \in V(G_{1}) \cap V(G_{2}); \\
	(1, \phi_{2}(x)) & \hspace{5mm}  x \in V(G_{2}) \setminus V(G_{1}).\\
	\end{array} \right.
	\end{align*}
	We may easily verify that $\psi$ is a proper coloring of $ G_{1} \cup G_{2}$ with color-set $ \{1, 2, ..., \chi(G_{1})\} \times  \{1, 2, ..., \chi(G_{2})\}$. Consequently, it follows that $\chi(G_{1} \cup G_{2}) \leqslant \chi(G_{1}) \times \chi(G_{2})$.
\end{proof} 
\vspace{1.5mm}
\noindent A consequence of the previous lemma is that, if we partition the edge-set of a graph $G$ into $E_{1}, E_{2},...,E_{k}$, then bounding the chromatic number of all spanning subgraphs $G_{i}$ of $G$ with edge-set
$E_{i}$ gives an upper bound for the chromatic number of $G$.\medbreak

\noindent Given a graph $G=(V,E)$ whose vertex-set is $V= V_{1} \cup V_{2}$, we may easily verify that $\chi(G) \leqslant \chi(G_{1})+\chi(G_{2})$, where $G_{1}$ and  $G_{2}$ are the subgraphs of $G$ induced by $V_{1}$ and $V_{2}$ respectively.\\

\noindent Let $G$ be a graph and let $L=x_1 x_2...x_n$ be an enumeration of its vertices. An edge $x_{i}x_{j}$ of $G$ is said to be a \textit{jump with respect to $L$} if $|i-j| > 1$. Two jumps $e=x_{l}x_{m}$ and $e'=x_{p}x_{q}$ of $G$ with $l<m$ and $p<q$  are said to be  \textit{secant edges with respect to  $L$}  if one of the following cases holds:
	\begin{description}[itemindent=1.5mm, leftmargin=1.5mm,  itemsep=1.5mm]
		\item[(i)] $l<p<m<q$;
		\item[(ii)]  $p<l<q<m$.
	\end{description} 

\noindent Let $D$ be an orientation of $G$. Two arcs $a$ and $a'$ of $D$ are said to be \textit{secant arcs with respect to $L$} if their corresponding edges in $G$ are so. An arc $a=(x_{i},x_{j})$ is said to be \textit{forward} with respect to $L$ if $i <j$. Otherwise, it is called \textit{backward} with respect to $L$.\medbreak

\noindent  Considering a graph $G$ having a linear  ordering of its vertices without secant edges, it is obvious that the restriction $L'$ of $L$ to the vertices of a subgraph $H$ of $G$ is  an ordering of $V(H)$  with no secant edges also. In view of this observation, we are able to color properly each graph having a linear ordering of its vertices with no secant edges: 
\begin{lemma}\label{nosecant} 
	Suppose that a graph $G$ admits an enumeration $L$ of its vertices such that $G$ has no secant edges with respect to $L$, then  $G$ is $3$-colorable. \end{lemma} 
\begin{proof}
	Let $L= x_{1}x_{2}...x_{n}$ be an enumeration of $V(G)$ with no secant edges. To reach our goal, we  prove  that $G$ is 2-degenerate.  Let $H$ be a subgraph of $G$ and let $L'$ be the restriction of $L$ to the vertices of $H$.  If  $H$ has no jumps with respect to $L'$, then $H$ is simply a disjoint union of paths and thus $\Delta(H) \leq 2$. Otherwise, we consider  a jump $x_{i}x_{j}$  of $H$ such that $|i-j|$ is minimal. Without loss of generality, we assume that $i<j$. Note that $N_{H}(x_{i+1}) \subseteq \{ x_{i}, x_{i+2}\}$, since otherwise   we get either secant edges or a jump smaller than $x_{i}x_{j}$, a contradiction. Thus $d_{H}(x_{i+1}) \leq 2$. In view of what precedes, it follows that   $G$ is $2$-degenerate. Consequently, we get that $G$ is $3$-colorable. This terminates the proof.  \end{proof}
\noindent The previous lemma will be  a central tool for the demonstration of the main theorems of the next sections. Furthermore, it will be used to  reinforce the conjecture that will be established  on the improvement of the upper bound of the chromatic number of strong digraphs having no subdivisions of two-blocks cycles.
\section{On bounding $\chi(S$-$Forb(C(k_{1},k_{2})) \cap \cal H \cal P)$ and $\chi(S$-$Forb(B(k_1, k_2; k_3))$ \\ $\cap \cal H)$}\label{c2section1sub2}
\noindent Denoting by $\cal H \cal P$ the class of digraphs having a Hamiltonian directed path,  Lemma \ref{nosecant} enables us to answer Addario et al.'s  question \cite{2path} positively  for the case of digraphs in $S$-$Forb(C(k_{1},k_{2})) \cap \cal H \cal P$: 

\begin{theorem}\label{2cyclehamil} Let $k_{1}$ and $ k_{2}$ be positive integers such that $k_{1} \geq k_{2} \geq 2$ and let $D$ be a digraph in $S$-$Forb(C(k_{1},k_{2})) \cap \cal H \cal P$. Then the chromatic number of $D$ is at most  $3k_1$.
\end{theorem}
\begin{proof}
	Set $P= x_{0}, x_{1}, ... ,x_{n}$ be a Hamiltonian directed path of $D$.  For $0 \leq i \leq k_1-1$, we define  $V_{i}=\{x_{i+ \alpha . k_1};$ $ \alpha \geq 0\}$. Let $D_{i}$ be the subdigraph of $D$ induced by $V_{i}$ and  let $P_{i}=x_{i_{1}}, x_{i_{2}},...,x_{i_{t}}$ be the restriction of $P$ on the vertices of  $D_{i}$  with $i_{1}<i_{2}<...<i_{t}$. Clearly, $\{V(D_0),V(D_1), ...., V(D_{k_1-1}) \}$ forms a partition of $V(D)$. Thus once we  color  each $D_{i}$ properly  by   $\chi(D_{i})$ colors, we  obtain a proper coloring of the whole digraph  using  $k_1. \chi(D_{i})$ distinct colors.  This means that, to  reach the desired goal,  it suffices to bound from above  the chromatic number of each $D_i$.  \begin{claim}\label{dinosecant}
		For all $0 \leq i \leq k_1-1$, 	$D_{i}$ has no secant arcs with respect to $P_{i}$. \end{claim} \noindent \sl {Subproof.} \upshape Assume the contrary is true and let $a=(x_{l},x_{m})$ and $a'=(x_{p},x_{q})$ be two secant arcs of  $D_{i}$ with respect to $P_{i}$. Here  there are four cases to  consider: $a$ and $a'$ are both forward or both backward, $a$ is forward and $a'$ is backward, or $a$ is backward and $a'$ is forward.   Assume first that $a$ and $a'$ are both forward, that is, $l < m$ and $p<q$. By symmetry, we  can assume  that $l<p<m<q$. Due to the definition of $D_{i}$, it is easy to see that $P[x_{l},x_{p}]$ and $P[x_{m},x_{q}]$ are of length at least $k_1-1$. Consequently, the union of $(x_{l},x_{m}) \cup P[x_{m},x_{q}]$ and $P[x_{l},x_{p}] \cup (x_{p},x_{q})$ forms a subdivision of $C(k_1,k_1)$ and so a subdivision of $C(k_1,k_2)$, a contradiction. The proof of the remaining cases is  similar to that of the previous one.  $\hfill  \lozenge$ \medbreak 
	\noindent According to  Claim \ref{dinosecant} and Lemma \ref{nosecant}, it follows that $\chi(D_{i}) \leq 3$ for all $0 \leq i \leq k_1-1$ and thus $\chi(D) \leq 3k_1$. This ends the proof. 
\end{proof}

\noindent As a consequence of the above theorem, we are able to find an upper bound for  the chromatic number of a  $C(k_{1},k_{2})$-subdivision-free digraph in terms of  the size of its maximal stable set:

\begin{corollary}	Let $k_{1} $ and $ k_{2}$ be two positive integers  and let $D$ be a  $C(k_{1},k_{2})$-subdivision-free digraph. Then the chromatic number of $D$  is at most $3 . \alpha(D).k$, where $k=max$ $\{k_1,k_2\}$.
\end{corollary}
\begin{proof}	Due to Gallai-Milgram theorem \cite{milgram}, the vertices of $D$ can be partitioned into  $\alpha(D)$ vertex-disjoint directed paths, where $\alpha(D)$ is the size of a maximal stable set of $D$. Let $P_{1}, P_{2}, ..., P_{\alpha(D)}$  be the disjoint directed paths  covering  $V(D)$  and let  $D_{i}$ be the subdigraph of $D$ induced by $V(P_{i})$ for all $i \in [\alpha(D)]$. It is easy to verify that  $D_{i}$ contains no subdivisions of $C(k_1,k_2)$ and $P_{i}$ is a Hamiltonian directed path of $D_{i}$. Thus Theorem   \ref{2cyclehamil} induces a proper $3k$-coloring of $D_i$. Consequently, by assigning  $3k$ distinct colors to each $D_i$, the required result follows.\end{proof}

\noindent Another important performance of both the notion of secant edges and Lemma \ref{nosecant}  appears in affirming Conjecture \ref{bispindle} for  the class  of Hamiltonian digraphs:
\begin{theorem}\label{bispindlehamil}  Let $D$ be a digraph in $S$-$Forb(B(k_1, k_2; k_3)) \cap \cal H$, where $k_1, k_2$ and $k_3$ are three  positive integers.  Then the chromatic number of $D$ is  at most $4k$, with $k=max$ $\{k_1, k_2, k_3\}$. \end{theorem}
\begin{proof} Let $k=max$ $ \{k_1, k_2, k_3\}$ and let $C=x_0, x_1,..., x_n, x_0$ be a Hamiltonian directed cycle of $D$. Let $L=x_0x_1...x_n$ be a linear ordering of the vertices of  the underlying graph $G$ of $D$ induced by $C$.
	For every $ 0 \leq i \leq k-1$, we set $G_i=G[\{x_{i+\alpha k};$ $ \alpha=0,1,2,3,... \}]$ and $G'_i=G_i \setminus \{x_i\}$. Let $L'_i$ be the ordering of $G'_i$ obtained by the restriction of $L$ on the vertices of $G'_i$.  \begin{claim}\label{nosecantdi} For all  $ 0 \leq i \leq k-1$, $G'_i$ has no secant edges with respect to $L'_i$. \end{claim} 
	\noindent \sl {Subproof.}  \upshape  Assume to the contrary that  $x_lx_m$ and $x_px_q$ are two secant edges  with respect to $L'_i$. Without loss of generality, assume that  $l<p<m<q$. Due to the definition of $G'_i$, note that   $l\geq k$, $p-l\geq k$, $m-p\geq k$ and $q-m\geq k$. 	To reach a contradiction to our assumption, we consider the  possible orientations of $x_lx_m$ and $x_px_q$. Assume first that  $(x_l,x_m)\in E(D)$.  If $(x_p,x_q) \in E(D)$, then the union of  $C[x_l,x_p] \cup (x_p,x_q)$,   $(x_l,x_m) \cup C[x_m,x_q]$ and $C[x_q,x_n] \cup (x_n,x_0) \cup C[x_0,x_l]$ forms a subdivision of  $B(k, k; k)$ and so  a subdivision of $B(k_1,k_2; k_3)$, a contradiction. Else if $(x_q,x_p) \in E(D)$, then the three internally disjoint directed paths $(x_q,x_p) \cup C[x_p,x_m]$, $C[x_q,x_n] \cup (x_n,x_0)  \cup C[x_0, x_l] \cup (x_l,x_m)$ and $C[x_m,x_q]$ form a subdivision of  $B(k, k; k)$ and so a subdivision of  $B(k_1, k_2; k_3)$, a contradiction. Thus $(x_l,x_m) \notin E(D)$ and hence $(x_m,x_l) \in E(D)$. If $(x_p,x_q) \in  E(D)$, then the three internally disjoint directed paths $C[x_p,x_m] \cup (x_m,x_l)$, $(x_p,x_q) \cup C[x_q,x_n] \cup (x_n,x_0) \cup C[x_0, x_l]$ and $C[x_l, x_p]$ form a subdivision of $B(k, k; k)$ and so a subdivision of $B(k_1, k_2; k_3)$. This  contradicts the fact that $D$ contains no subdivisions of $B(k_1, k_2; k_3)$ and implies that  $(x_q,x_p)\in E(D)$. But the  three internally disjoint directed paths $(x_m,x_l) \cup C[x_{l},x_p]$, $C[x_m,x_q] \cup( x_q,x_p)$ and $C[x_p,x_m]$ form a subdivision of $B(k_1, k_2; k_3)$,  a contradiction to our initial assumption.  This proves that $G'_i$ has no secant edges with respect to $L'_i$. $\hfill  \lozenge$ \medbreak
	\noindent	Due to Lemma \ref{nosecant} together with Claim \ref{nosecantdi}, we get   that $\chi (G'_i)\leq 3$. But $\chi (G_i)\leq \chi (G'_i) + 1$, then $\chi (G_i)\leq 4$ and so $ \chi (G)\leq \sum _{i=0}^{k-1} \chi (G_i)\leq  4k$. By considering  the fact that $\chi (D)=\chi(G)$, the required result follows.  \end{proof}
\noindent  The bound given in Theorem \ref{bispindlehamil} is certainly not the best possible. However, a better  bound is provided for the digraphs in $S$-$Forb(B(k_1, 1; k_3)) \cap \cal H$:

\begin{proposition}  Let $D$ be a digraph  in $S$-$Forb(B(k_1, 1; k_3)) \cap \cal H$, where $k_1$ and $ k_3$ are arbitrary positive integers. Then the    chromatic number of $D$ is at most  $2k-1$, with   $k=max$ $\{k_1, k_3\}$. 
\end{proposition}
\begin{proof}  Let $C=x_0, x_1, x_2,..., x_{n-1}, x_0$  be a Hamiltonian directed cycle of $D$ and let  $x$ be a vertex of $D$.  Without loss of generality, we may assume that $x=x_0$. If $n-k\leq k-1$, then $n\leq 2k-1$ and thus $\chi(D)\leq n\leq 2k-1$. So we may suppose that $n-k\geq k$. Assume that there exists an integer $t$ such that $ k\leq t \leq n-k$ and $x_t\in N(x)$. If $(x_0,x_t)\in D$, then the union of the three internally disjoint directed paths $C[x_0,x_t]$, $(x_0,x_t)$ and $C[x_t,x_0]$ is a subdivision of $B(k_1, 1; k_3)$, a contradiction. Thus $(x_t,x_0)\in D$ and so  the three internally disjoint directed paths $C[x_t,x_0]$, $(x_t,x_0)$ and $C[x_0,x_t]$ form a subdivision of $B(k_1, 1; k_3)$, a contradiction. This gives that $N(x) \subseteq \{x_1,x_2..., x_{k-1},x_{n-1}, x_{n-2},...,x_{n-(k-1)}\}$. Whence, the maximum degree of $D$ is at most $2k-2$ and so the chromatic number  of $D$ is at most $2k-1$. \end{proof} 
\section{On the  improvement of  $\chi(S$-$Forb(C(k_{1},k_{2})) \cap \cal S)$}\label{chapter2section2}
This section is devoted to introduce an elegant approach which is supposed to be contributory in improving the chromatic number of digraphs in $S$-$Forb(C(k_{1},k_{2})) \cap \cal S$  from $12k^2$ to $4k^2$, given that $k= max  \{k_1,k_2\}$.  Once our  approach is completely verified, it leads not just to ameliorate the best known upper bound for the chromatic number of strong digraphs not containing two-blocks cycles,  but to prove a more general statement. This is because we are treating the  problem of the existence of two-blocks cycles in  the class of digraphs having a spanning out-tree, that includes the strong ones.\\

\noindent  At first, we start with some basic definitions, standard notations  and  preliminary results  that will be essential for the  coming proofs. \medbreak

\noindent Among the most effective tools that play a major role in constructing our proofs are the trees.  Let $G$ be a graph with a spanning tree  $T$ rooted at  $r$. For a vertex $x$ of $G$, there is a unique $rx$-path in $T$, denoted by $T[r,x]$. The \textit{level} of $x$ with respect to $T$, denoted by $l_{T}(x)$, is the length of this path. The \textit{ancestors} of $x$ are the vertices that belong to $T[r,x]$. For an ancestor $y$ of $x$, we write $y \leqslant_{T} x$. Conversely, we denote by $S(x)$ the set of the vertices $v$ of $G$ such that  $x$ is an ancestor of $v$.  The subtree of $T$ rooted at $x$ and induced by $S(x)$ is denoted by $T_{x}$. An ancestor $y$ of $x$ is said to be its \textit{predecessor} if $l_{T}(y)=l_{T}(x)-1$.  Two leaves of $T$ are said to be \textit{sisters} if they share the same predecessor in $T$.  We say that $T$ is  \textit{normal} in $G$ if for every edge $xy$ of $G$ either $x\leqslant_{T} y$ or vice versa. It is well known that every connected graph has a normal spanning tree with any preassigned root.\medbreak

\noindent Because we are concerned in the study of digraphs rather than graphs, similar definitions are introduced for oriented trees. Recall that an \textit{out-tree} is an oriented tree in which all the  vertices have in-degree at most 1.  Given a digraph $D$ having a spanning  out-tree $T$ with source $r$, the \textit{level} of a vertex $x$  with respect to $T$, denoted by $l_{T}(x)$, is the length of the unique $rx$-dipath in $T$. For a non-negative integer $i$, we define $L_{i}(T)=\{x \in V(T);$ $ l_{T}(x)=i\}$.  For a vertex $x$ of $D$, the \textit{ancestors} of $x$ are the vertices that belong to $T[r,x]$. For two vertices $x_{1}$ and $x_{2}$ of $D$, the \textit{least common ancestor} $y$ of $x_{1}$ and $x_{2}$  is the common ancestor of $x_{1}$ and $x_{2}$  having the highest level in $T$. Note that the latter notion is well-defined since $r$ is a common ancestor of all the vertices of $D$. An arc $(x,y)$ of $D$ is said to be forward with respect to $T$ if $l_{T}(x) < l_{T}(y)$, otherwise it is called a backward arc.  We say that $T$ is a \textit{maximal out-tree} of $D$, if $y$ is an ancestor of $x$ for every backward arc $(x,y)$ of $D$.\medbreak
\noindent  The next proposition shows an interesting structural property on   digraphs having a spanning out-tree:
\begin{proposition}\label{yarab}
	Given a digraph $D$ having a  spanning out-tree $T$, then $D$ contains a maximal out-tree.
\end{proposition}
\begin{proof}
	Initially, set $T_{0}:=T$. If $T_{0}$ is maximal, there is nothing to do.  Otherwise there is an arc $(x,y)$ of $D$ which is backward with respect to $T_{0}$ such that $y$ is not ancestor of $x$.  Let $T_{1}$ be the out-tree obtained from $T_{0}$ by adding  $(x,y)$  to $T_{0}$, and deleting the arc of head $y$ in $T_{0}$. We can easily see that the level of each vertex in $T_{1}$ is at least its level in $T_{0}$, and there exists a vertex ($y$) whose level strictly increases. Since the level of a vertex cannot increase infinitely, we can see that after a finite number of repeating the above process  we reach an out-tree  which is maximal.
\end{proof}
\noindent As a generalization of the notion of secant edges defined with respect to an enumeration of the vertices of a given graph $G$, we develop this concept as  follows: Given that $P=x_{1}, x_{2}, ..., x_{n}$ is a path of  $G$, two edges  $e$ and $e'$  of $G$ are said to be  \textit{secant edges with respect to $P$} if they are secant with respect to the enumeration induced by $P$. However for  a normal spanning tree $T$  of $G$ rooted at $r$, two edges $e$ and $e'$ of $G$  are said to be \textit{secant edges  with respect to $T$} if $e$ and $e'$ are secant edges with respect to a path $P$ in $T$ starting at $r$.\medbreak

\noindent In the light of what Lemma \ref{nosecant}  induces on the coloring of graphs admitting a linear ordering of their vertices without secant edges, we get a proper coloring for each graph having a Hamiltonian path with no secant edges: 

\begin{proposition}\label{Tpath} Let $G$ be a graph having a Hamiltonian path $P$  such that $G$ has no secant edges with respect to $P$, then the chromatic number of $G$ is at most $3$. \end{proposition}

\noindent Since  Hamiltonian  paths   can be viewed as  special types of  normal spanning trees, a question that naturally arises here is, can similar result be established for any normal spanning tree?   Many endeavors and efforts  made to answer this question enables us  to predict that the following is true: 

\begin{conjecture}\label{4colorable} If $G$ is a graph with a normal spanning tree $T$ such that $G$ has no secant edges with respect to $T$, then the chromatic number of $G$ is at most $4$. \end{conjecture}

\noindent In the next section, we affirm  the above conjecture  for some trees.  However,  the rest of this section  is dedicated to show how proving Conjecture \ref{4colorable} lead  to improve the chromatic number of digraphs in $S$-$Forb(C(k_{1},k_{2})) \cap \cal S$.\medbreak

\noindent Denoting by  $\cal T$ the  tree classes for which Conjecture \ref{4colorable} is  confirmed,  we are able to get a better upper bound  for the chromatic number of digraphs  containing no subdivisions of a two-blocks cycle $C(k_{1},k_{2})$ and having a spanning out-tree, relying  only  on the simple notion of a maximal out-tree and the strategies of leveling and digraph decomposing:

\begin{theorem}\label{improvement}  Let $k_1$ and $ k_2$ be positive integers with $k_1 \geq k_2$  and let $T$ be an  out-tree whose underlying graph is  in  $\cal T$. Then every digraph  $D$  in $S$-$Forb(C(k_{1},k_{2}))$ and having a spanning out-tree $T$ is colored properly  using $4.k_1.(k_2-1)$ colors.\end{theorem}

\begin{proof}	Without loss of generality, we may assume that $T$ is a maximal out-tree in $D$. This assumption is true due to Proposition \ref{yarab}.	For $0 \leq i \leq k_1-1$, we set $V_{i}=\cup_{\alpha \geqslant 0} L_{i+\alpha k_1}(T)$ and we define $D_{i}$ to be the subdigraph of $D$ induced by $V_{i}$. Then we partition the arcs of $D_{i}$ as follows:$$A_{1}=\{(x,y); \hspace{1mm} x\leqslant_{T} y \hspace{1mm} or \hspace{1mm} y \leqslant_{T} x\};$$
	$$A_{2}=A(D_{i}) \setminus A_{1}.$$
	For $0 \leq i \leq k_1-1$ and $j=1,2$, let $D_{i}^{j}$ be the spanning subdigraph of $D_{i}$ whose arc-set is $A_{j}$.
	
	\begin{claim}\label{hel}
		$\chi(D_{i}^{1}) \leq 4$ for all $i \in \{0, 1, ..., k_1-1\}$.
	\end{claim}
	\noindent {\sl {Subproof.}}  Let $(D_{i}^{1})^{T}$ be the digraph obtained by the union of $D_{i}^{1}$ and $T$. Consider the underlying graph $G_{i}^{1}$ of $(D_{i}^{1})^{T}$. Observe that the underlying tree $T^1$ of $T$ is a normal tree of $G_{i}^{1}$. Moreover,  $G_{i}^{1}$ has no secant edges with respect to $T^1$, since otherwise $D$ contains a subdivision of $C(k_1,k_1)$ and so a subdivision of $C(k_1,k_2)$, a contradiction. According to the definition of  $\cal T$ and to the fact that $T^1$ is in $\cal T$, it follows that  $\chi(G_{i}^{1}) \leq 4$ and so $\chi ((D_{i}^{1})^{T}) \leq 4$. Because $D_{i}^{1}$ is a subdigraph of $(D_{i}^{1})^{T}$, our desired claim yields. $\hfill {\lozenge}$
	\begin{claim}\label{am}
		$\chi(D_{i}^{2}) \leq k_2-1$ for all $i \in \{0, 1, ..., k_1-1\}$.
	\end{claim}
	\noindent \sl {Subproof.} \upshape Suppose to the contrary that there exists $i_{0}$ such that $\chi(D_{i_{0}}^{2}) \geq k_2$. By  Theorem \ref{directedpath}, $D_{i_{0}}^{2}$ contains a directed path $P$ of order at least $k_2$, say  $P= y_{0}, y_{1}, ...y_{k_2-2}, y_{k_2-1}$. Note that the arcs of $D_{i_{0}}^{2}$ are of the form $(x,y)$ such that $x$ is not an ancestor of $y$ and vice versa. Because $T$ is a maximal  out-tree in $D$, it follows that  all the arcs of $D_{i_{0}}^{2}$ are forward with respect to $T$.	For all $ 0 \leq j \leq k_2-2$, observe that  $y_{j}$ is not an ancestor of $y_{k_2-1}$. To this end, we assume that the contrary is true. Let $\alpha$ be the greatest index such that $y_{\alpha} \leqslant_{T} y_{k_2-1}$. By the definition of $A_{2}$, $y_{\alpha +1}$ is distinct from $y_{k_2-1}$. Let $x$ be the least common  ancestor of $y_{\alpha}$ and $y_{\alpha +1}$. Then the union of $T[x,y_{\alpha +1}] \cup P[y_{\alpha +1}, y_{k_2-1}]$ and $T[x,y_{k_2-1}]$ forms a subdivision of $C(k_1,k_1)$, and thus a subdivision of $C(k_1,k_2)$, a contradiction.
	Now consider $x$ to be the least common ancestor of $y_{0}$ and $y_{k_2-1}$.  Due to the above observation, we remark  that $x$ is distinct from $y_{0}$ and  $T[x,y_{k_2-1}] \cap P=\{y_{k_2-1}\}$. This implies that $T[x,y_{k_2-1}]$ and $T[x,y_{0}] \cup P$ are two  internally disjoint $(x,y_{k_2-1})$-dipaths of length at least  $k_1$ and $k_2$ respectively, and so their union is a subdivision of $C(k_1,k_2)$ in $D$, which contradicts the fact that $D$ is $C(k_1,k_2)$-subdivision-free. This verifies  our claim. $\hfill {\lozenge}$\medbreak
	
	\noindent Therefore, applying Lemma \ref{far} to Claim \ref{hel},  Claim  \ref{am} and  the fact that $D_{i}= D_{i}^{1} \cup D_{i}^{2}$, we get that $\chi(D_{i}) \leq 4(k_2-1)$ for all $i \in \{0,1, ...,k_1-1\}$. Consequently, as $V(D_{i})$ form a partition of $V(D)$ for $0 \leq i \leq k_1-1$, we obtain a proper $4.k_1.(k_2-1)$-coloring of $D$. This completes  the proof. \end{proof}

\noindent Since any strong digraph contains a spanning out-tree, Theorem \ref{improvement} induces  the following improvement on the upper bound given by Kim et al. \cite{kim}  for some strong $C(k_1,k_2)$-subdivision-free digraphs:
\begin{corollary}
	Let $D$ be a digraph in $S$-$Forb(C(k_{1},k_{2})) \cap \cal S$, given that $k_1$ and $ k_2$ are positive integers with $k_1 \geq k_2$.  If the underlying graph of a spanning out-tree $T$ of $D$ is  in  $\cal T$, then the chromatic number of $D$ is at most   $4.k_1.(k_2-1)$.
\end{corollary}

\section{A contribution to Conjecture \ref{4colorable}}\label{chapter2section3}

The aim of this  section is to support  Conjecture \ref{4colorable}. First, we start with some definitions and terminologies that formulate our conjecture in a more approachable way. Then, we introduce some  trees class  for which  our conjecture is  confirmed. \\

\noindent The next definition  allows us to deal more flexibly  with Conjecture \ref{4colorable}:
\begin{definition}	Let $G$ be a graph with a spanning rooted tree $T$. We say that $G$ is saturated with respect to $T$,  if the following three conditions are satisfied:
	\begin{description}[itemindent=1.5mm, leftmargin=1.5mm,  itemsep=1.5mm]
		\item[(i)] $T$ is a normal tree in $G$.
		\item[(ii)] $G$ contains no secant edges with respect to $T$.
		\item[(iii)]  $ \forall$ $x,$ $ y \in V(G)$ such that  $x\leqslant_{T} y$, if $xy \notin E(G)$ then  $G+xy$ contains secant edges with respect to $T$.
	\end{description}
\end{definition}
\medskip

\noindent Given a tree  $T$  in a graph $G$, $G[T]$ denotes the subgraph of $G$ induced by $V(T)$. When $G[T]$ is saturated with respect to $T$, we simply say that $G[T]$ is saturated. If $T$ is a normal tree in $G[T]$,  we denote by $G^{*}[T]$  a \textit{saturation} of $G[T]$, that is,  $G^{*}[T]$ is the graph obtained from $G[T]$  by adding  the maximum number of edges so that $T$ remains a normal tree of the obtained graph $G^{*}[T]$ and $G^{*}[T]$ has no secant edges with respect to $T$. \medbreak

\noindent From now on, we assume that  $G$ is a  graph  defined as in Conjecture \ref{4colorable}, unless otherwise specified.  Without loss of generality, we may assume that $G[T]$ is saturated, since otherwise we add to $G$ the maximal number of edges so that $T$ remains a normal tree of the obtained graph $G^{*}[T]$ and $G^{*}[T]$ has no secant edges with respect to $T$. Given an edge $yz$  of $G$, we say that $yz$ is a \textit{jump  with respect to $T$}, if $y$ is neither the predecessor of $z$ in $T$ nor a successor of it. For a jump $yz$  of $G$ with $y \leqslant_{T}  z$, we say that $y$ is the \textit{lower end} of $yz$ and $z$ is its \textit{upper end}. For a vertex $x$ of $G$,  $yz$ is called a \textit{jump  over  $x$}  if $y \leqslant_{T} x \leqslant_{T} z$. Furthermore, we say that  $yz$ is  a \textit{minimal jump} of $G$ if there is no other jump  whose both ends lie in  $T[y,z]$. However, $yz$  is called a \textit{higher jump} of $G$ if  the level of $y$ in $T$ is not smaller than that of the lower end of  any other jump.  \medbreak

\noindent For a non-negative integer $i$, star$^i$-like trees are well-organized trees  for which  Conjecture \ref{4colorable} is partially confirmed. Initially, a \textit{star$^0$-like tree}  is  a tree having exactly one vertex  whose  degree  is strictly greater than $2$,  called a \textit{node}.  A \textit{star$^1$-like tree} is a tree obtained from a star$^0$-like tree by replacing each of its leaves by a star$^0$-like tree. Inductively, for $i \geq 2$,  a \textit{star$^i$-like tree} is defined to be  a tree obtained from a star$^{(i-1)}$-like tree by replacing each of its leaves by a star$^0$-like tree. A \textit{whip} is a  star$^0$-like tree in which the successors of its unique node are leaves. \medbreak

\noindent   The following well-known lemma  will be used frequently throughout this section:
\begin{lemma}\label{cutset}
	Let $G$ be a graph with a clique cut-set $S$ whose removal  decomposes the vertex-set of $G\setminus S$ into the connected components $X_{1}, X_{2}, ..., X_{k}$. Then 
	$$\chi(G) \leq  \underset{1 \leq i \leq k}{max} \chi(G_{i});$$
	where $G_{i}$ is the subgraph of $G$ induced by $X_{i} \cup S$ for $i=1,2, ...,k$.
\end{lemma} 
\noindent Note that the subgraphs $G_{1}, G_{2}, ...G_{k}$ mentioned above are called the \textit{blocks of decomposition} of $G$ with respect to the cut-set $S$.\medbreak

\noindent Now we are ready to demonstrate   Conjecture \ref{4colorable} for the case of saturated graphs having either a star$^0$-like tree or a  star$^1$-like tree as their  spanning normal trees. Before delving into the proofs, we need the following  obvious result that establishes a proper $3$-coloring of $G$ in case its spanning tree   is a   whip: 

\begin{proposition}\label{whip} Let $T$ be a whip, then $G$ has a proper $3$-coloring through which  the leaves are uniquely colored.\end{proposition}
\begin{proof} Consider the path $P$  obtained from $T$ by contracting the leaves into a single vertex $v$, and let $G'$ be the graph induced by $V(P)$.   One may easily verify that $G'$ has no secant edges with respect to the Hamiltonian path $P$. According to Proposition \ref{Tpath}, it follows that $\chi(G') \leq 3$. Assume without loss of generality that the color of $v$ is 3. This means that all the leaves are not adjacent in $G$  to any vertex of color $3$. Thus, we can extend the proper $3$-coloring of $G'$ to a proper $3$-coloring  of $G$ by assigning the color $3$ to the leaves. This ends the proof. \end{proof}

\begin{proposition}Let $T$ be a star$^0$-like tree, then $\chi(G[T]) \leq 4$.\end{proposition}
\begin{proof}The proof is by induction on the number of vertices. Let $x$ be the unique vertex of $T$ such that $d_T(x) >2$. We consider two cases:  $G[T_x]$ is either a tree or not.  Assume first that $G[T_x]$ is not a tree. This implies  that $G[T_x]$  has a jump $yz$ with $x \leqslant_{T} y \leqslant_{T} z$. It can be easily verified that  $S=\{y,z\}$ is a clique cut-set of $G[T]$, where $G_1=G[T[y,z]]$ and $G_2=G[T-T[y,z]+\{y,z\}]$ are the two blocks of decomposition of $G[T]$ with respect to $S$. Due to Proposition \ref{Tpath}, it follows that $\chi(G_1) \leq 3$ and so $\chi(G_1) \leq 4$. However, the induction hypothesis gives that $\chi(G_2) \leq 4$. Consequently, due to Lemma \ref{cutset}, we get that $\chi(G) \leq 4$. To complete the proof, it remains to consider the case where  $G[T_x]$ is  a tree, that is, $G[T_x]=T_x$. Contract the vertices of $T_x$ to a single vertex $v$. The obtained contracted graph $G'$ has a Hamiltonian path ending at $v$ without secant edges. According to Proposition \ref{Tpath}, $G'$ has a proper $3$-coloring.  Since any tree is $2$-colorable, we can extend the coloring of $G'$ to a proper $4$-coloring of $G$, by assigning alternatively to each vertex in $T_x$ either the color of $v$  or the color $4$.  This yields the required result. \end{proof} 

\begin{lemma} \label{overnode} Let $T$ be a star$^1$-like tree such that every jump is over a node, then $\chi(G[T]) \leq 4$. \end{lemma}

\begin{proof} We argue by induction on the number of vertices. Assume that the contrary is true, and let $x$ be the node of minimal level, $x_1, x_2, ..., x_m$ be all the other nodes and $y_1, y_2, ..., y_m$ be the successors of $x$.  For all $1  \leq i \leq m$,  suppose that $G[T_{y_i}]$ is a tree and consider the whip  $T'$ obtained from $T$ by contracting each $T_{y_i}$ to a single vertex $v_{y_i}$. Clearly, the resultant graph  $G'$ has no  secant edges with respect to its normal spanning tree $T'$. Note that  the  leaves of $T'$ are the  vertices $\{v_{y_1}, v_{y_2}, ..., v_{y_m}\}$. According to  Proposition \ref{whip}, $G'$ has a proper $3$-coloring that attributes to the leaves  a unique  color, say the color $3$. Since any tree is $2$-colorable, we may extend the coloring of $G'$ to a proper $4$-coloring of $G$ by assigning appropriately to each vertex in $G[T_{y_i}]$  the color $3$ or $4$. This implies that $\chi(G[T]) \leq 4$, which   contradicts  our assumption. Thus,  there exists  $ i_0 \in  [m]$ such that  $x_{i_0}$ is a node    over which there is a minimal jump, say $ab$. Without loss of generality, assume that $ a \leqslant_{T} b$. We choose $ab$ so that the level of $a$ is maximal.  Because $G$ is saturated and $ab$ is a minimal higher jump over $x_{i_0}$, we may easily check that $b$ is a successor of $x_{i_0}$ which is in its turn a successor of $a$. For all $1 \leq j \leq r$,  let $b_j$ be the successor of $x_{i_0}$ and  $d_j$ be the last vertex in $T_{x_{i_0}}$ such that $b_j \leqslant_{T} d_j$ and $ad_j \in E(G)$.  Again, due to  the maximality of the level of $a$ and the fact that $G$ is saturated,  we get that $av \in E(G)$ for all $v$ in $T[b_j,d_j]$ and for all $j \in [r]$. Assume first that $b_j=d_j$ for every  $ j \in [r]$, and let $t$ be the unique predecessor of $a$. If $t$ is adjacent to $b$, then $t$ is adjacent to $b_j$ for every $j \in [r]$, since otherwise we get either secant edges, jumps over a vertex of degree $2$ in $T$, or a neighbor of $a$ in $G[T_{x_{i_0}}] $ other than $x_{i_0}$ and $b_j$, a contradiction to our assumption. Let $T'$ be the tree obtained from $T$ by deleting $x_{i_0}$ and adding the edges $ab_j$ for all $j \in \{1,2,...,r\}$. It is straightforward to see that  $T'$ is a star$^1$-like tree, $G'=G[T']$ is a graph with  no secant edges with respect to $T'$ and  every jump of $G'$ is over a node.  Applying the induction hypothesis to $G'$, we get that $\chi(G') \leq 4$. But the facts that  $N_G(x_{i_0}) \subseteq N_G(t)$ and  $t x_{i_0} \notin E(G)$ induce  a proper $4$-coloring of $G$ by assigning  to $x_{i_0}$ the same color of $t$,   a contradiction. Thus $t b_j \notin E(G)$ for all $ j \in [r]$. This implies that, for every $j \in \{1,2,...,r\}$, $G + t b_j$ contains secant edges with respect to $T$, and so $G$ contains either a neighbor of $a$ other than $b_j$ and $x_{i_0}$ in $G[T_{y_i}]$, or $a$ has a neighbor in $G[T-T_t]$. Due to the assumption that $b_j=d_j$ for all $j \in [r]$, it follows that $a$ has a neighbor $u$ in $G[T-T_t]$. We choose $u$ so that it is the nearest neighbor of $x_{i_0}$  to the root. Note that $u$ is probably  the root.  It is  easy to see that $u$ is adjacent to $b_j$ for all $j \in [r]$.  Consider now the graph $G'=G[T']$, where $T'$  is the tree obtained from $T$ by deleting $x_{i_0}$ and adding the edges $ab_j$ for all $j \in \{1,2,...,r\}$. Proceeding in the same manner as before, we reach a contradiction to the assumption that $\chi(G) > 4$.  Therefore, it follows that there exists $ 1 \leq j \leq r$ such that $b_j \neq d_j$, say $j=1$. Let $p_1$ be the predecessor of $d_1$ and let $p_2$ be that of $p_1$. Note that probably $p_1 = b_1$, and so in this case $p_2=x_{i_0}$. Consider the graph $G'=G[T']$, where $T'$  is the tree obtained from $T$ by deleting $p_1$ and adding the edge $p_2 d_1$. Clearly, $T'$ is a star$^1$-like tree, $G'$ has no secant edges with respect to its normal spanning tree $T'$ and the jumps of $G'$ are only over the vertices whose degree in  $T'$ is greater than $	2$. By the induction hypothesis, we get that $\chi(G') \leq 4$. To reach the final contradiction, we extend the proper $4$-coloring of $G'$ to a proper $4$-coloring  of $G$ by assigning to $p_1$ a color distinct from the colors of its three unique neighbors $\{d_1, p_2, a\}$, a contradiction. This terminates the proof.    \end{proof}

\begin{proposition}
	Let $T$ be a star$^1$-like tree, then $\chi(G[T] ) \leq 4$.
\end{proposition}
\begin{proof}
	We proceed by induction on the number of  vertices of $G$. There are two cases to consider: If every jump of $G$ is over a node, then Lemma \ref{overnode} implies that  $\chi (G) \leq 4$. Otherwise,  if  $G$ has a jump $yz$ such that all the vertices  in $T[y,z]$ are not nodes, let  $T_1=T[y,z]$ and $T_2$ be the tree obtained from $T$ by replacing the path $T[y,z]$ by the edge $yz$.  Whence, $S=\{y,z\}$ is a clique cut-set of $G$, with $G_1=G[T_1]$ and $G_2= G[T_2]$ are the two blocks of decomposition of $G$ with respect to  $S$. Observe that $G_1$ is a cycle and so $\chi(G_1) \leq 3$. However, $T_2$ is a star$^1$-like tree and  $G_2$ is a graph with no secant edges with respect to its normal spanning  $T_2$. By applying the induction hypothesis, we get that  $\chi(G_2 ) \leq 4$. Therefore, due to Lemma \ref{cutset}, the hoped result directly follows. 
\end{proof}

\noindent The question that arises here is, can a similar result  be established  for $\chi(G[T])$ if $T$ is a star$^i$-like tree with $i \geq 2$? If  the answer to this question is  positive, then Conjecture \ref{4colorable} is confirmed and consequently Theorem \ref{improvement} turns to be true  whatever is  the spanning out-tree.    To see the latter statement, we need the following observation   with the fact that any subgraph has a chromatic number not exceeding that of the graph it contains: 

\begin{proposition}
	Let $G$ be a graph defined as in Conjecture \ref{4colorable}. Then there exists a graph $G'$ with normal spanning tree $T'$, such that $G'$ has no secant edges with respect to $T'$, $T'$ is a star$^{i}$-like tree for some non-negative integer $i$, $E(G') \setminus E(T') = E(G) \setminus E(T)$ and the  vertices $a',b'$ of $G'$  corresponding to the vertices $a,b$ of  $G$ preserve the tree-order, i.e  if $a \leqslant_{T} b$ then  $a' \leqslant_{T'} b'$.
\end{proposition}

\begin{proof}
	The proof is built by induction on the number of vertices. Let $G$ be a graph  defined as in Conjecture \ref{4colorable}, and  let $L$ be the set of leaves of $T$. For sake of simplicity, for any graph $G'$ having a  normal spanning tree $T'$ such that $G'$ has no secant edges with respect to $T'$, where $T'$ is a star$^{i}$-like tree for some non-negative integer $i$, $E(G') \setminus E(T') = E(G) \setminus E(T)$ and the  vertices $a',b'$ of $G'$  corresponding to the vertices $a,b$ of  $G$ preserve the tree-order, we say that $G'$ satisfies the property $\color{frenchblue} \circledast$.  Assume that there exists $b \in L$ such that $b$ has no sisters. Consider $T_1$ the tree obtained from $T$ by contracting $b$ and its unique predecessor $a$  into a new vertex $v$. It is easy to verify that $G_1=G[T_1]$ is a graph with no secant edges with respect to its normal spanning tree $T_1$. By applying the induction hypothesis, it follows  the existence of a graph  $G'_1$ with normal spanning tree $T'_1$ satisfying the property  $\color{frenchblue} \circledast$. Let $v'$ be the vertex corresponding to $v$ in $G'_1$. If $v'$ is a leaf in $T'_1$, we simply un-contract $v'$. Else, we remove  the neighbors of $v'$ which  are neighbors of $b$ in $G-T$ and we add them to a successor $b'$ of $v'$.  It is not hard to check that resultant graph $G'$ has  the property $\color{frenchblue} \circledast$. This verifies the case where $G$ has at least one leaf $b$ with no sisters. To complete the demonstration, it remains to prove the case where every vertex in $L$ has sisters. To this end, let $b$ be a vertex in $L$ such that $b$ has sisters, say $\{b_1, b_2, ..., b_m\}$ with $m \geq 1$. Assume that  $a$ is the unique predecessor of $b$ and its sisters in $T$. Let $T_1$ be the tree obtained from $T$ by contracting  $b$  and its sisters into a new vertex $v$. Clearly, $G_1=G[T_1]$ is a graph  defined  as  in Conjecture \ref{4colorable}. Thus,  the induction hypothesis induces the existence of a graph $G'_1$ with the property $\color{frenchblue} \circledast$.  More precisely, $G'_1$ has a normal spanning tree $T'_1$ with no secant edges with respect to $T'_1$, such that $T'_1$  is a  star$^{i}$-like tree for some  non-negative integer $i$. Assume that $v'$ is the vertex of $G'_1$ that corresponds to $v$ and $a'$ is its unique predecessor.  Since  $E(G'_1) \setminus E(T'_1) = E(G_1) \setminus E(T_1)$, it follows that  $a'$ matches with  $a$.  Un-contracting $v'$, we denote by  $\{b', b'_1, b'_2, ..., b'_t\}$ the set of  the resultant vertices that correspond  to $\{b,b_1, b_2, ..., b_t\}$ respectively. If $a'$ is a node in $T'_1$, we obtain the desired graph by adding to $b'$ and to each $b'_i$  a copy of ${T'_1}_{v'}$. Else if $d_{T'_1} (a')=2$, we proceed as before by adding  to $b'$ and to each $b'_i$  a copy of ${T'_1}_{v'}$. This step produces a tree $T''_1$ which is not a star$^{i}$-like tree for any $i$. Thus to reach the tree  we are looking for, we add to each leaf in $T''_1 -{{T''_1}_{a'}}$ a whip.  This leads to  a graph $G'$ having a  normal spanning tree $T'$, such that $G'$ has no secant edges with respect to $T'$, $T'$ is a star$^{i+1}$-like tree, $E(G') \setminus E(T') = E(G) \setminus E(T)$ and the  vertices $a',b'$ of $G'$  corresponding to the vertices $a,b$ of  $G$ preserve the tree-order. This ends the proof. 
\end{proof}
\section{Subdivisions of oriented cycles in tournaments}

\noindent Given a subdivision $H$ of an oriented  cycle $C$, a block of $H$ is said to be \textit{dilated} if its length in $H$ is strictly  greater than that of  its corresponding one in $C$. Otherwise, it is called a \textit{non-dilated} block. For example, $C(3,1,2,1)$ is a subdivision of the four-blocks cycle $C(1,1,1,1)$ with exactly two non-dilated blocks. 
 
\begin{proposition}
Let $k_1, k_2, ...,k_{2m}$ be positive integers with $m \geq 2$. If $k_i+k_{i+1} \geq 3$ for all $i \in \{1, 3, ..., 2m-1\}$, then every tournament $T$ of order $m+\sum_{i=1}^{2m} k_i$ contains a subdivision of $C(k_1,k_2,...,k_{2m})$, with at least $m$ non-dilated blocks. 
\end{proposition}
\begin{proof} Let $T_1, T_3, ..., T_{2m-1}$ be subtournaments of $T$ of order  $k_1+k_{2}+1$, $k_3+k_{4}+1$, ...., $k_{2m-1}+k_{2m}+1$, respectively.  Due to Theorem \ref{path}, it follows that $T_i$ contains a two-blocks path $P(k_i,k_{i+1})$  for all $i \in \{1, 3, ..., 2m-1\}$, which is the  union of two disjoint directed paths $Q_{i}$ and $Q_{i+1}$ of respective length $k_{i}$ and $k_{i+1}$ which are disjoint except in their initial vertex.   We denote by $x_i$ the terminal vertex of $Q_{i}$ for all $ 1 \leq i \leq 2m$.  By considering the possible orientations of the edges $x_{2i} x_{2i+1  (mod \hspace{1mm } 2m)}$ for all $1 \leq i \leq m$, we get the desired subdivision.  \end{proof}

	\end{document}